\documentclass{amsart}
\usepackage{epsfig}
\usepackage[english]{babel}
\usepackage{bigints}
\usepackage{esint}
\usepackage{nccmath}
\usepackage{relsize}
\usepackage{amssymb}
\usepackage{amsmath}
\usepackage{graphicx}
\usepackage{mathabx}

\newcommand{\R}{\mathbb{R}}
\newcommand{\E}{\mathbb{E}}
\newcommand{\1}{\mathbb{1}}
\renewcommand{\P}{\mathbb{P}}

\newcommand{\N}{\mathbb{N}}
\newcommand{\W}{\mathbb{W}}

\newcommand{\ce}{\mathbb E}

\newcommand{\ds}{\displaystyle}

\newtheorem{thank}{\ \ \ Acknowledgment}
\newcounter{tictac}

\def\1{\,\rlap{\mbox{\small\rm 1}}\kern.15em 1}

\def\build#1_#2^#3{\mathrel{\mathop{\kern 0pt#1}\limits_{#2}^{#3}}}
\def\tend#1#2{\build\hbox to 12mm{\rightarrowfill}_{#1\rightarrow #2}^{ }}

\def\tendn{\tend{n}{\infty}}
\def\converge#1#2#3#4{\build\hbox to
#1mm{\rightarrowfill}_{#2\rightarrow #3}^{\hbox{\scriptsize #4}}}

\theoremstyle{definition}

\newtheorem{thm}{Theorem}[section]
\newtheorem{prop}[thm]{Proposition}
\newtheorem{lemm}[thm]{Lemma}

\newtheorem{Prop}[thm]{Proposition}
\newtheorem{que}[thm]{Question}

\newtheorem{rems}[thm]{Remarks}

\newtheorem*{xrem}{Remark}

\newcommand{\Prod}{\mathop{\mathlarger{\mathlarger{\mathlarger{\prod}}}}}

\def\Pr{\mathbb P}
\newcommand{\beq}{\begin{equation}}
\newcommand{\eeq}{\end{equation}}

\begin{document}
\title{On the spectral type of some class of rank one flows}
\author{E.\ H. \ {El} Abdalaoui }
\address{ Universit\'e de Rouen-Math\'ematiques \\
  Labo. de Maths Raphael SALEM  UMR 60 85 CNRS\\
Avenue de l'Universit\'e, BP.12 \\
76801 Saint Etienne du Rouvray - France . \\}
\email{elhoucein.elabdalaoui@univ-rouen.fr}

\maketitle



{\renewcommand\abstractname{Abstract}
\begin{abstract}
It is shown that a certain class of Riesz product type measure on $\R$ is singular. This proves the 
singularity of the spectral types 
of some class of rank one flows. Our method is based on the extension of 
the Central Limit Theorem approach to the real line which gives a  
new extension of Salem-Zygmund Central Limit Theorem.  
\vspace{8cm}\\

\hspace{-0.7cm}{\em AMS Subject Classifications} (2010): 37A15, 37A25, 37A30, 42A05, 42A55.\\

\end{abstract}
\thispagestyle{empty}
\newpage
\section{Introduction}\label{intro}
The purpose of this paper is to study the spectral type of some class of rank one flows. Rank one flows have simple spectrum and using a random Ornstein procedure \cite{Ornstein}, 
A. Prikhod'ko 
in \cite{prikhodko-orn} produce a family of mixing
rank one flows. It follows that the mixing 
rank one flows may possibly contain a candidate for the flow version of the 
Banach's well-known problem whether there exists a dynamical flow $(\Omega,{\mathcal
{A}},\mu,(T_t)_{t \in \R})$ with simple Lebesgue spectrum \footnote{Ulam in his book \cite[p.76]{Ulam} stated the Banach problem in the following form 
\begin{que}[Banach Problem]
Does there exist a square integrable function $f(x)$ and a measure preserving transformation $T(x)$, 
$-\infty<x<\infty$, such that the sequence of functions $\{f(T^n(x)); n=1,2,3,\cdots\}$ forms a complete 
orthogonal set in Hilbert space?
\end{que}}. In \cite{prikhodko}, 
A. Prikhod'ko introduced a class of rank one flows called {\it {exponential staircase rank one flows}} and state 
that in this class the answer to the flow version of Banach problem is affirmative. Unfortunately, as we shall 
establish, this is not the case since the spectrum of a large class of 
exponential staircase rank one flow is singular and this class contain a subclass of  
Prikhod'ko examples.

Our main tools are on one hand an extension to $\R$ of the 
CLT method (introduced in \cite{elabdaletds} for the torus) and on the other hand
the generalized Bourgain methods \cite{Bourgain} obtained in \cite{elabdal-archiv}
(in the context of the Riesz products on $\R$).

This allows us to get a new extension of the 
Salem-Zygmund CLT Theorem \cite{Zygmund} to the trigonometric sums 
with \textbf{real frequencies}.

Originally Salem-Zygmund CLT Theorem concerns the asymptotic stochastic behaviour 
of the lacunary trigonometric sums on the torus.
Since Salem-Zygmund pioneering result, the central limit theorem for
trigonometric sums has been intensively studied by many authors,  Erd\"os \cite{Erdos},
  J.P. Kahane \cite{Kahane}, J. Peyri\`ere \cite{Peyriere}, Berkers \cite{Berkes}, Murai \cite{Murai},
Takahashi \cite{Takahashi}, Fukuyama and  Takahashi
\cite{Fukuyama}, and many others. The same method is used to
study the asymptotic stochastic behaviour of Riesz-Raikov sums \cite{petit}.
Nevertheless all these results concern only the trigonometric sums on the torus.

Here we obtain the same result on $\R$.
The fundamental ingredient  in our proof is based on the famous 
Hermite-Lindemann Lemma  in the transcendental number theory \cite{Waldschmidt}.

Notice that the main argument used in the torus case \cite{elabdaletds} 
is based on the density of trigonometric polynomials. This argument cannot be applied here since
the density of trigonometric polynomials in $L^1(\R,\omega(t) dt)$ ($\omega$ is a positive function in $L^1(\R)$), 
is not verified unless $\omega$ satisfies some extra-condition. 
Nevertheless, using the density of the functions 
with compactly supported Fourier transforms, we are able to conclude.

The paper is organized as follows. In section 2, we review some standard facts from the
spectral theory of dynamical flows. In section 3, we recall the basic construction of the   
rank one flows obtained by the cutting and stacking method and we state our main result. In section 4, 
 we summarize and extend the relevant material on the
Bourgain criterion concerning the  singularity of the generalized Riesz products on $\R$.
In section 5, we develop the CLT method for trigonometric sums with real frequencies and 
we prove our main result concerning the singularity of a exponential 
staircase rank one flows.

\section{Basic facts from spectral theory of dynamical flows }\label{sec:1}
 A dynamical flow  is  a quadruplet $(X,\mathcal{A},\mu,(T_t)_{t \in \R})$
where $(X,\mathcal{A},\mu)$ is a Lebesgue probability space
 and $(T_t)_{t\in \R}$ is a measurable action of the group $\R$ by measure preserving transformations. (It means that 
 \begin{itemize}\item
 each $T_t$ is a bimeasurable invertible transformation of the probability space such that, for any
$A \in \mathcal{A}$, $\mu(T_t^{-1}A)=\mu(A)$,
\item for all $s,t\in\R$, $T_s\circ T_t = T_{s+t}$,
\item the map $(t,x)\mapsto T_t(x) 
$ is measurable from $\R\times X$ into $X$.)
\end{itemize}

Let us recall some classical definitions. A dynamical flow is \emph{ergodic} if
every measurable set which is invariant under all the maps $T_t$ either has measure zero or one. A number
$\lambda$ is an  \emph{eigenfrequency} if there exists nonzero function $f \in L^2(X)$ such that, for all $t\in\R$,
$f \circ T_t=e^{i \lambda t} f$. Such a function $f$ is called an  \emph{eigenfunction}.
An ergodic flow $(X,\mathcal{A},\mu,(T_t)_{t \in \R})$ is  \emph{weakly mixing} if every eigenfunction is constant (a.e.).
A flow $(X,\mathcal{A},\mu,(T_t)_{t \in \R})$ is  \emph{mixing} if for all $f,g \in L^2(X)$,
\[\bigintss f \circ T_t(x) \overline{g}(x) d\mu(x) \tend{|t|}{+\infty} \bigintss f(x) d\mu(x) \bigintss \overline{g}(x) d\mu(x).\]

Any dynamical flow $(T_t)_{t \in \R})$ induces an action of $\R$ by unitary operators acting on
$L^2(X)$ according to the formula  $U_{T_t}(f)=f \circ T_{-t}$. When there will be no ambiguity on the choice of the flow, we will denote $U_t=U_{T_t}$.

The \emph{spectral properties} of the flow are the property attached to the unitary representation associated to the flow. We recall below some classical facts; for details and references see \cite{CFS} or \cite{Handbook-1B-11}.

Two dynamical flows
$(X_1,\mathcal{A}_1,\mu_1,(T_t)_{t \in \R})$ and $ (X_2,\mathcal{A}_2,\mu_2,(S_t)_{t \in \R})$ are \emph{metrically isomorphic} if there exists a measurable map $\phi$ from $(X_1,\mathcal{A}_1,\mu_1)$ into $(X_2,\mathcal{A}_2,\mu_2)$, with the following properties:
\begin{itemize}
\item $\phi$ is one-to-one,
 \item For all $A \in \mathcal{A}_2$, $\mu_1(\phi^{-1}(A))=\mu_2(A).$
\item $S_t \circ \phi=\phi \circ T_t$, $\forall t \in \R$.
\end{itemize}
If two dynamical flows $(T_t)_{t \in \R}$ and $(S_t)_{t \in \R}$ are metrically isomorphic then the isomorphism $\phi$ induces
an isomorphism $V_{\phi}$ between the Hilbert spaces $L^2(X_2)$ and $L^2(X_1)$ which acts according to the formula
$V_{\phi}(f)=f \circ \phi$. In this case, since
$V_{\phi}U_{S_t}=U_{T_t}V_{\phi}$, the adjoint groups $(U_{T_t})$ and $(U_{S_t})$ are unitary equivalent. Thus
if two dynamical flows are metrically isomorphic then the corresponding adjoint groups of unitary operators are unitary
equivalent. It is well known that the converse statement is false \cite{CFS}.

By Bochner theorem, for any $f \in L^2(X)$, there exists a unique finite Borel measure $\sigma_f$ on $\R$ such that
\[\widehat{\sigma_f}(t)=\bigintss_{\R} e^{-it\xi}\ d\sigma_f(\xi)=\langle U_tf,f \rangle=
\bigintss_{X} f \circ T_t(x)\cdot \overline{f}(x) \ d\mu(x).\]
$\sigma_f$ is called the \emph{spectral measure} of $f$. If $f$ is eigenfunction with eigenfrequency $\lambda$ then
the spectral measure is the Dirac measure at $\lambda$.

The following fact derives directly from the definition of the spectral measure: let $(a_k)_{1\leq k\leq n}$ be complex numbers and $(t_k)_{1\leq k\leq n}$ be real numbers; consider $f\in L^2(X)$ and denote $F=\sum_{k=1}^n a_k \cdot f\circ T_{t_k}$. Then the spectral measure $\sigma_F$ is absolutely continuous with respect to the spectral measure $\sigma_f$ and
\begin{equation}\label{radon-spectral}
\frac{d\sigma_F}{d\sigma_f}(\xi)=\left|\sum_{k=1}^n a_k e^{it_k\xi}\right|^2.
\end{equation}

Here is another classical result concerning spectral measures : let $(g_n)$ be a sequence in $L^2(X)$, converging to $f\in L^2(X)$ ; then the sequence of real measures $(\sigma_{g_n}-\sigma_f)$ converges to zero in total variation norm.
\medskip

The \emph{maximal spectral type} of $(T_t)_{t \in \R}$ is the equivalence class of Borel
measures $\sigma$ on $\R$ (under the equivalence relation $\mu_1
\sim \mu_2$ if  and only if $\mu_1<<\mu_2$ and $\mu_2<<\mu_1$),
such that
 $\sigma_f<<\sigma$ for all $f\in L^2(X)$ and
if $\nu$ is another measure for which $\sigma_f<<\nu$
for all $f\in L^2(X)$ then $\sigma << \nu$.

The maximal spectral type is realized as the spectral measure of one function: there exists $h_1\in L^2(X)$
such that $\sigma_{h_1}$ is in the equivalence class defining the maximal spectral type of $(T_t)_{t \in \R}$.
By abuse of notation, we will call this measure the maximal
spectral type measure. 

The reduced maximal type $\sigma_0$ is the
maximal spectral type of ${(U_{t})}_{t \in \R}$ on $L_0^2(X)\stackrel{\rm
{def}}{=}\left\{f \in L^2(X)~:~ \displaystyle \bigintss f d\mu=0 \right\}$. The
spectrum of $(T_t)_{t \in \R}$ is said to be discrete (resp. continuous, resp.
singular, resp. absolutely continuous , resp. Lebesgue) if
$\sigma_0$ is discrete (resp. continuous, resp. singular with respect to Lebesgue measure, resp.
absolutely continuous with respect to Lebesgue measure).

The cyclic space of $h \in L^2(X)$ is
\[Z(h) \stackrel{\rm {def}}{=} \overline {{\rm {span}} \{U_{t}h\,:\,t \in \R \} }.
\]

There exists an orthogonal decomposition of $L^2(X)$ into cyclic spaces
\beq\label{rozkladsp} L^2(X)=\bigoplus_{i=1}^\infty Z(h_i),\;\;
\sigma_{h_1}\gg\sigma_{h_2}\gg\ldots \eeq Each
decomposition~(\ref{rozkladsp}) is be called a {\em spectral
decomposition} of $L^2(X)$ (while the sequence of measures is called a {\em
spectral sequence}). A spectral decomposition is unique up to equivalent class of
the spectral sequence. The spectral decomposition is determined
by the maximal spectral type and the
{\em multiplicity function}
$M:\R\to\{1,2,\ldots\}\cup\{+\infty\}$, which is
defined $\sigma_{h_1}$-a.e. by $ M(s)=\sum_{i=1}^\infty
1_{Y_i}(s)$, where $Y_1=\R$ and
$Y_i=supp\,\frac{d\sigma_{x_i}}{d\sigma_{x_1}}$ for $i\geq2$.

The flow has {\em simple spectrum} if
$1$ is the only essential value of $M$. The multiplicity is
{\em homogeneous} if there is only one essential value of $M$.
The essential supremum of $M$ is called the {\em maximal
spectral multiplicity}.

Von Neumann showed that the flow ${(T_t)_{t \in \R}}$ has homogeneous Lebesgue spectrum if and only if the
associated group of unitary operators ${(U_{t})_{t \in \R}}$ satisfy the Weyl commutation relations for some
one-parameter group $(V_t)_{t \in \R}$ i.e.
\[
U_{t} V_s=e^{-ist}V_sU_{t},~~~~~~~~~s,t \in \R,
\]
where $e^{-ist}$ denotes the operator of multiplication by $e^{-ist}$.\\ It is easy to show that
the Weyl commutation relations implies that the  maximal spectral type is
invariant with respect to the translations. 
The proof of von Neumann homogeneous Lebesgue spectrum theorem can be found in \cite{CFS}.

\section{Rank one flows by Cutting and Stacking method}\label{CSC}

Several approach of the notion of rank one flow have been proposed in the literature. The notion of \emph{approximation of a flow by periodic transformations} has been introduced by Katok and Stepin in \cite{Kat-Step} (see Chapter 15 of \cite{CFS}). This was the first attempt of a definition of a rank one flow. 

In \cite{Deljunco-Park}, del Junco and Park adapted the classical Chacon construction \cite{lebonchacon} to produce
similar construction for a flow. The flow obtain by this method is called the Chacon flow.

This cutting and stacking construction has been extended by Zeitz (\cite{zeitz}) in order to give a general definition of a rank one flow. In the present paper we follow this cutting and stacking (CS) approach and we recall it now.
We assume that the reader is familiar with the CS construction of a rank one map acting on certain
measure space which may be finite or $\sigma$-finite. A nice account may be founded in~\cite{Friedman}.

Let us fix a sequence $(p_n)_{n\in\N}$ of integers $\geq2$ and a sequence of finite sequences of non-negative real numbers $\left({(s_{n,j})}_{j=1}^{p_{n-1}}\right)_{n>0}$.

Let ${\overline{B_0}}$ be a rectangle of height $1$ with horizontal base $B_0$. At stage one divide $B_0$ into $p_0$ equal 
parts $(A_{1,j})_{j=1}^{p_0}$. Let $(\overline{A}_{1,j})_{j=1}^{p_0}$ denotes the flow towers over
$(A_{1,j})_{j=1}^{p_0}$. In order to construct the second flow tower,
put over each tower $\overline{A}_{1,j}$ a rectangle spacer of height $s_{1,j}$ (and base of same measure as $A_{1,j}$) and form a stack of height $h_{1}=p_0 +\sum_{j=1}^{p_0}s_{1,j}$ in the usual
fashion. Call this second tower $\overline{B_1}$, with $B_1=A_{1,1}$.

At the $k^{th}$ stage, divide the base $B_{k-1}$
of the tower ${\overline{B}_{k-1}}$ into $p_{k-1}$ subsets $(A_{k,j})_{j=1}^{p_{k-1}}$ of equal measure.
Let $(\overline{A}_{k,j})_{j=1}^{p_{k-1}}$ be the towers over
$(A_{k,j})_{j=1}^{p_{k-1}}$ respectively. Above each tower $\overline{A}_{k,j}$, put a rectangle spacer of height $s_{k,j}$ (and base of same measure as $A_{k,j}$). Then form a stack of height $h_{k} = p_{k-1}h_{k-1} + \sum_{j=1}^{p_{k-1}}s_{k,j}$ in the usual
fashion. The new base is $B_k=A_{k,1}$ and the new tower is $\overline{B_k}$.

All the rectangles are equipped with Lebesgue two-dimensional measure that will be denoted by $\nu$. Proceeding this way we construct what we call a rank one flow
${(T_t)_{t \in \R}}$ acting on a certain measure space $(X,{\mathcal B} ,\nu)$ which may
be finite or $\sigma-$finite depending on the number of spacers added at each stage. \\
This rank one flow will be denoted by

$$(T^t)_{t \in \R} \stackrel{\text{def}}{=}  
\left(T^t_{(p_n, (s_{n+1,j})_{j=1}^{p_{n}})_{n\geq0}}\right)_{t \in \R}$$

The invariant measure $\nu$ will be finite if and only if 
$$\displaystyle \sum_{k=0}^{+\infty}
\frac{\sum_{j=1}^{p_{k}} s_{k+1,j}}{p_kh_k}<+\infty.$$ 
\noindent{}In that case, the measure will be normalized in order to have a probability.

\begin{rems} The only thing we use from  \cite{zeitz} is the definition of rank one flows. Actually
a careful reading of Zeitz paper \cite{zeitz} shows that the
author assumes that for any rank one flow there exists always at least one time $t_0$ such that 
$T_{t_0}$ has rank one property. But, it turns out that this is not the case in general as proved by 
Ryzhikhov in 
\cite{RyzhikovWCT}. Furthermore,
if this property was  satisfied 
then the weak closure theorem for flows would hold as a direct consequence of the King weak closure theorem
($C(T_{t}) \subset C(t_0) \stackrel{\rm {J.King}}{=}WCT(T_{t_0})\subset WCT({T_t}) \subset C(T_{t})$, 
where $C(t_0)$ is the centralizer of $T_{t_0}$ and 
$WCT(T_{t_0})$ is the weak closure of $T_{t_0}$).
\end{rems}

\subsection{Exponential staircase rank one flows of type I}
The main issue of this note is to study 
the spectrum of a subclass of a rank one flows called 
{\it exponential staircase rank one flows of type I} which are defined as follows.

Let $(m_n,p_n)_{n \in \N}$ be a sequence of positive integers such that 
$m_n$ and $p_n$ goes to infinity as $n$ goes to infinity. Let $\varepsilon_n$ be a sequence of 
rationals numbers which converge to $0$. Put
$$ 
\omega_n(p)=\frac{m_n}{\varepsilon_n^2}p_n\Big(\exp\big(\frac{\varepsilon_n.p}{p_n}\big)-1\Big)
{\rm {~~for~any~~}} p \in \{0,\cdots,p_n-1\},
$$   
and define the sequence of the spacers $((s_{n+1,p})_{p=0,\cdots,p_n-1})_{n \geq 0}$ by
$$
 h_n+s_{n+1,p+1}=\omega_n(p+1)-\omega_n(p),~~~p=0,\cdots,p_n-1, n \in \N.
$$
In this definition we assume that
\begin{enumerate}
 \item $m_n \geq \varepsilon_n.h_n$, for any $n \in \N$.
\item $\displaystyle \frac{\log(p_n)}{m_n} \tendn 0$ if $p_n \geq \frac{m_n}{\varepsilon_n}$
\item $\displaystyle \frac{\log(p_n)}{m_n} \tendn 0$ and  $\displaystyle \frac{\log(p_n)}{p_n} \leq \varepsilon_n$ 
if $p_n < \frac{m_n}{\epsilon_n}$
\end{enumerate}
We will denote  this class of rank one flow by

$$(T^t)_{t \in \R} \stackrel{\text{def}}{=}  
\left(T^t_{(p_n,\omega_n)_{n\geq0}}\right)_{t \in \R}.$$
It is easy to see that this class of flows contain a large class of examples introduced by Prikhodko \cite{prikhodko}. Indeed,
assume that $h_n^{\beta} \geq p_n \geq h_n^{1+\alpha}$, $\beta \geq 2$ and $\alpha \in ]0,\frac14[$. Then, by  assumption 
(1), we have
$$\frac{\log(p_n)}{m_n} \leq \beta \frac{\log(h_n)}{\varepsilon_n h_n},$$
Taking $\beta=\varepsilon_n \big(\lfloor h_n^{\delta} \rfloor+1\big)$, $0<\delta<1$, we get 
$$
\frac{\log(p_n)}{m_n}  \tendn 0.
$$   
 
In \cite{elabdal-archiv}  it is proved that
the spectral type of of any rank one flow $(T^t_{(p_n, (s_{n+1,j})_{j=1}^{p_{n}})_{n\geq0}})_{t \in \R}$ is given by some kind of Riesz-product measure on $\R$. 
To be more precise, the authors  in \cite{elabdal-archiv} proved the following theorem 

\begin{thm}[Maximal spectral type of rank one flows]\label{max-spectr-typ}
For any $s\in(0,1]$, the spectral measure $\sigma_{0,s}$ is the weak
 limit of the sequence of probability measures 
$$\prod_{k=0}^{n}|P_k(\theta)|^2 K_s(\theta)\,d\theta,$$
\noindent{} where
\begin{eqnarray*}
&&P_k(\theta)=\frac 1{\sqrt{p_k}}\left(
\sum_{j=0}^{p_k-1}e^{{i\theta(jh_k+\bar s_{k}(j))}}\right),~~\bar s_{k}(j)=\sum_{i=1}^js_{k+1,i}, 
~\bar s_{k}(0)=0 \nonumber.  \\
\nonumber
\end{eqnarray*}
\noindent{} and
$$K_s(\theta)=\frac{s}{2\pi}\cdot{\left(\frac{\sin(\frac{s\theta}2)}{\frac{s\theta}2}\right)^2}.$$

In addition the continuous part of spectral type of the rank one flow is equivalent 
to the continuous part of $\ds \sum_{k \geq 1} 2^{-k}\sigma_{0,\frac{1}{k}}.$
\end{thm}
The theorem above gives a new generalization of Choksi-Nadkarni Theorem \cite{Nadkarni1}, \cite{Nadkarni4}. 
We point out that in \cite{elabdal}, the author generalized the Choksi-Nadkarni Theorem 
to the case of funny rank one group actions for which the group is compact and Abelian.\\
 
We end this section by stating our main result.
\begin{thm}\label{main}
Let $(T^t)_{t \in \R}=
\left(T^t_{(p_n,\omega_n)_{n\geq0}}\right)_{t \in \R}$
 be a exponential staircase rank one flow of type I
associated to  
\begin{eqnarray}
\nonumber \omega_n(p)=\frac{m_n}{\varepsilon_n^2}p_n \exp\big({\frac{\varepsilon_n.p}{p_n}}\big), 
~~p=0,\cdots,p_n-1. 
\end{eqnarray}  Then the spectrum of $(T_t)_{\R}$ is singular.
\end{thm}

\section{ On the Bourgain singularity criterium of generalized Riesz products on $\R$ }
In this section, for the convenience of the reader we repeat the relevant material from \cite{elabdal-archiv}
without proofs, thus making our exposition self-contained. 
Let us fix $s\in(0,1)$ and denote by $\mu_s$ the probability measure of density $K_s$ on $\R$, that is, 
$$
d\mu_s(\theta)=K_s(\theta)\,d\theta
$$
We denote by $\sigma$ the spectral measure of $\1_{\overline{B_{0,s}}}$ given as the weak limit of the following
generalized Riesz products
\begin{eqnarray}\label{sptrkone}
d\sigma={\rm {W-}}\lim_{N\to+\infty}\prod_{k=1}^{N}|P_k|^2 d\mu_s,
\end{eqnarray}
\noindent{} where
\begin{eqnarray*}
&&P_k(\theta)=\frac 1{\sqrt{p_k}}\left(
\sum_{j=0}^{p_k-1}e^{{i\theta(jh_k+\bar s_{k}(j))}}\right),~~\bar s_{k}(j)=\sum_{i=1}^js_{k+1,i}, 
~\bar s_{k}(0)=0 \nonumber.  \\
\nonumber
\end{eqnarray*}
Let us recall the following Bourgain criterion established in \cite{elabdal-archiv}. 
\begin{thm}[Bourgain criterion]\label{Bourg-cri}
The following are equivalent
\begin{enumerate}
\item [(i)]  $\sigma$ is singular with respect to Lebesgue measure.

\item[(ii)]  {\it $\inf \left \{\displaystyle  \displaystyle \bigintss_\R
\prod_{\ell=1}^L\left| {P_{n_\ell}}\right|\; d\mu_s\;:\; L\in
{\N},~n_1<n_2<\ldots <n_L\right\}=0.$ }
\end{enumerate}
\end{thm}

As noted in \cite{elabdal-archiv}  
to prove the singularity of the spectrum of the rank one flow it is sufficient to prove that a 
weak limit point of the sequence $\left(\left||P_m|^2-1\right|\right)$ is bounded by below 
by a positive constant. More precisely, the authors in \cite{elabdal-archiv} established the following
proposition. 

\begin{Prop}\label{CDsuffit} Let $E$ be an infinite set of positive integers. 
Suppose that there exists a constant $c>0$ such that, for any positive function $\phi \in L^2(\R,\mu_s)$, 
$$\liminf_{\overset{m\longrightarrow +\infty}{m\in E}} \bigintss_{\R}\phi\left||P_m|^2-1\right|
\;d\mu_s \geq c \bigintss_{\R}\phi \;d\mu_s.$$
Then $\sigma$ is singular.
\end{Prop}

The proof of the proposition \ref{CDsuffit} is based on the following lemma.

\begin{lemm}\label{limsup} Let $E$ be an infinite set of positive integers. Let $L$ be a positive integer and  $
0\leq n_1<n_2<\cdots<n_L$ be integers. Denote $Q=\prod_{\ell=1}^L\left | {P_{n_\ell}}\right|$. Then
 \begin{eqnarray*}
  \displaystyle \limsup_{\overset{m\longrightarrow +\infty}{m\in E}}\bigintss  Q \left| P_m\right|\;d\mu_s
\leq 
  \displaystyle \bigintss Q \;d\mu_s -\frac 18\left(
\liminf_{\overset{m\longrightarrow +\infty}{m\in E}}
\displaystyle \bigintss  Q \left| \left| P_m\right| ^2-1\right|
\;d\mu_s\right) ^2. 
\end{eqnarray*}
\end{lemm}

\noindent{}For sake of completeness we recall from \cite{elabdal-archiv} 
the proof of the proposition \ref{CDsuffit}. 
\begin{proof}[Proof of Proposition \ref{CDsuffit}]${}$\\
Let $\displaystyle \beta=\inf\left\{\bigintss Q\;
d\mu_s~:~Q=\prod_{\ell=1}^L\left | {P_{n_\ell}}\right|, L \in \N,
0\leq n_1<n_2<\cdots<n_L\right\}$. Then, for any such $Q$, we have
\begin{eqnarray*}
    \bigintss Q
\;d\mu_s \geq \beta\quad\text{and}\quad \liminf \bigintss Q |P_m| \;d\mu_s \geq
\beta.
\end{eqnarray*}

\noindent{} Thus by Lemma \ref{limsup} and by taking the infimum over all $Q$ we get
\begin{eqnarray*}
\beta \leq \beta-\frac18(c\beta)^2
\end{eqnarray*}
\noindent It follows that
\[
\beta=0,
\] 
and the proposition follows from Theorem \ref{Bourg-cri}.
\end{proof}


\section{the CLT method for trigonometric sums and the singularity of the spectrum 
of exponential staircase rank one flows of type I}
The main goal of this section is to prove the following proposition
\begin{prop}\label{key-sasha} Let $(T^t)_{t \in \R}=
\left(T^t_{(p_n,\omega_n)_{n\geq0}}\right)_{t \in \R}$
 be a exponential staircase rank one flow of type I
associated to  
\begin{eqnarray}
\nonumber \omega_n(p)=\frac{m_n}{\varepsilon_n^2}p_n \exp\big({\frac{\varepsilon_n.p}{p_n}}\big), 
~~p=0,\cdots,p_n-1.
\end{eqnarray} Then, there exists a constant $c>0$ such that, 
for any positive function $f$ in $L^2(\R,\mu_s)$, we have 
$$\liminf_{m\longrightarrow +\infty} \bigintss_{\R}f(t)\left||P_m(t)|^2-1\right|
\;d\mu_s(t) \geq c \bigintss_{\R}f(t) \;d\mu_s(t).$$ 
\end{prop}
The proof of the proposition \ref{key-sasha} is based on the
study of the stochastic behaviour of the sequence $ |P_m|$.
For that, we follow the strategy introduced in \cite{elabdaletds} based on the method 
of the Central Limit Theorem for trigonometric sums.

This methods takes advantage of the following classical expansion 
\[
\exp(ix)=(1+ix)\exp\Big(-\frac{x^2}2+r(x)\Big), 
\]
\noindent{}where $|r(x)| \leq |x|^3$, for all real number $x$ \footnote{this is a direct consequence 
of Taylor formula with integral remainder.}, combined with some ideas developed in the proof 
of martingale central limit theorem due to McLeish \cite{mcleish}. Precisely, the main ingredient 
is the following theorem proved in \cite{elabdaletds}.

\begin{thm}\label{Mcleish}
Let $\{X_{nj}:1 \leq j \leq
k_n, n\geq 1\}$ be a triangular array of random variables and $t$ a real number. Let 
$$\displaystyle S_n=\sum_{j=1}^{k_n}X_{nj},~~~~~~~\displaystyle T_n=\prod_{j=1}^{k_n}(1+itX_{nj}),$$
and $$\displaystyle U_n=\exp \left (-\frac{t^2}2\sum _{j=1}^{k_n}
X_{nj}^2+\sum_{j=1}^{k_n}r(tX_{nj})\right ).$$ Suppose that
\begin{enumerate}
\item $\{T_n\}$ is uniformly integrable.
\item $\E(T_n) \tendn 1$.
\item $\displaystyle \sum_{j=1}^{k_n} X_{nj}^2 \tendn 1$ in probability.
\item $\ds \max_{1 \leq j \leq k_n} |X_{nj}|\tendn 0$ in probability.
\end{enumerate}
Then $\E(\exp(it S_n))\longrightarrow \exp(-\displaystyle\frac{t^2}2).$ 
\end{thm}
We remind that the sequence  $\{X_n, n\geq 1\}$ of random  variables is said to be uniformly integrable if and only 
if 
\[
\lim_{c \longrightarrow +\infty}\bigintss_{\big\{|X_n| >c\big \}} \big|X_n\big | d\Pr =0 ~~~
{\textrm {uniformly~in~}} n.   
\]
and it is well-known that if 
\begin{eqnarray}\label{uniform}
\sup_{ n \in \N}\bigg(\ce\big(\big|X_n\big|^{1+\varepsilon}\big)\bigg) < +\infty,
\end{eqnarray}
for some $\varepsilon$ positive, then $\{X_n\}$ is uniformly integrable.\\

Using Theorem \ref{Mcleish} we shall prove the following extension to $\R$ of Salem-Zygmund CLT theorem,
which seems to be of independent interest.
\begin{thm}\label{Gauss}
Let $A$ be a Borel subset of $\R$ with $\mu_s(A)>0$ and 
let $(m_n,p_n)_{n \in \N}$ be a sequence of positive integers such that 
$m_n$ and $p_n$ goes to infinity as $n$ goes to infinity. Let $\varepsilon_n$ be a sequence of 
rationals numbers which converge to $0$ and
$$ 
\omega_n(j)=\frac{m_n}{\varepsilon_n^2}p_n\exp\big(\frac{\varepsilon_n.j}{p_n}\big)
{\rm {~~for~~any~~}} j \in \{0,\cdots,p_n-1\}.
$$ 
Assume that
\begin{enumerate}
 \item $m_n \geq \varepsilon_n.h_n$, for any $n \in \N$.
\item $\displaystyle \frac{\log(p_n)}{m_n} \tendn 0$ if $p_n \geq \frac{m_n}{\varepsilon_n}$
\item $\displaystyle \frac{\log(p_n)}{m_n} \tendn 0$ and  $\displaystyle \frac{\log(p_n)}{p_n} \leq \varepsilon_n$ 
if $p_n < \frac{m_n}{\epsilon_n}.$
\end{enumerate}   
Then, the distribution of the
sequence of random variables
$\frac{\sqrt{2}}{\sqrt{p_n}}\sum_{j=0}^{p_n-1} \cos(\omega_n(j)t)$
converges to the Gauss distribution. That is, for any real number $x$, we have 

\begin{eqnarray}{\label{eq:eq6}}
\frac1{\mu_s(A)}\mu_s\left \{t \in A~~:~~\frac{\sqrt{2}}{\sqrt{p_n}}
\sum_{j=0}^{p_n-1} \cos(\omega_n(j)t) \leq x \right \}\nonumber \\
\tend{n}{\infty}\frac1{\sqrt{2\pi}}\bigintss_{-\infty}^{x}
e^{-\frac12t^2}dt\stackrel{\rm {def}}{=}{\mathcal {N}}\left( \left] -\infty,x \right] \right) .
\end{eqnarray}
\end{thm}

\noindent{}We start by proving the following proposition.

\begin{prop}\label{w-lim}
 There exists a subsequence of
the sequence $\left (\left |\left| P_n(\theta)\right|^2-1 \right|\right)_{n\geq \N}$
which converges weakly in $L^2(\R,\mu_s)$ to some non-negative function
$\phi$ bounded above by $2$.
\end{prop}

\noindent{}In the proof of Proposition \ref{w-lim}
we shall need the following lemma proved in \cite{elabdal-archiv}.

\begin{lemm}\label{Key1} The sequence of probability measures $|P_n(\theta)|^2 K_s(\theta)\,d\theta$ converges weakly
to  $K_s(\theta)\,d\theta$.
\end{lemm}

\begin{proof}[Proof of Proposition \ref{w-lim}]{}Since for all $n$ , 
we have $\left\|\P_n\right\|^2_{L^2(\mu_s)}=1$. Therefore, 
the sequence $\left(\left|\left|P_n\right|^2-1\right|\right)_{n\in \N}$ 
is bounded in $L^2\big(\R,\mu_s\big)$, thus admits 
a weakly convergent subsequence. Let us denote by $\phi$ one such weak limit function.
Let $f$ be a bounded continuous function on $\R$. We have
$$
\bigintss 
f\cdot\left|\left|P_n\right|^2-1\right| d\mu_s 
\leq \bigintss f\cdot \left|P_n\right|^2  d\mu_s+\bigintss f\, d\mu_s.$$
By Lemma \ref{Key1} we get
$$
\lim_{n\to+\infty} \bigintss f\cdot\left|\left|P_n\right|^2-1\right| d\mu_s 
\leq 2 \bigintss f\; d\mu_s.
$$
Hence, for any bounded continuous function $f$ on $\R$, we have
$$ 
 \bigintss f\cdot \phi\; d\mu_s 
\leq 2 \bigintss f \;d\mu_s.
$$
which proves that the weak limit $\phi$ is bounded above by 2.
\end{proof}

Let us prove now that the function $\phi$ is bounded by below by a universal positive constant. For that 
we need to prove Proposition \ref{key-sasha}.
Let $n$ be a positive integer and put

\begin{eqnarray*}
\W_n {\stackrel {\rm {def}}{=}}&&  \Big \{ {\sum_{j \in
I}\eta_j \omega_n(j)} ~~:~~
\eta_j \in \{-1,1\}, I \subset \{0,\cdots,p_n-1\}\Big \}.
\end{eqnarray*}

\noindent{} The element $w =\sum_{i \in
I}\eta_j \omega_n(j)$ is called a word.\\
 
We shall need the following two combinatorial lemmas. The first one is a classical result 
in the transcendental number theory and it is due to Hermite-Lindemann.\\

\begin{lemm}[{\bf Hermite-Lindemann, 1882}]\label{Hermite}
 Let $\alpha$ be a non-zero algebraic number. Then, the number $\exp(\alpha)$ is transcendental. 
\end{lemm}

\noindent{}We state the second lemma as follows.

\begin{lemm} For any $n \in \N^*.$
 All the words of $\W_n$ are distinct.
\end{lemm}
\begin{proof}{}
Let $w,w' \in \W_n$, write
\begin{eqnarray*}
w &=&\sum_{j \in I}\eta_j \omega_n(j),\\
w' &=&\sum_{j \in I'}\eta'_j \omega_n(j).
\end{eqnarray*}
\noindent{}Then $w=w'$ implies
\[
\sum_{j \in I}\eta_j \omega_n(j)-\sum_{j \in I'}\eta'_j \omega_n(j)=0
\]
Hence
\[
\sum_{j \in I}\eta_j \exp(\frac{\varepsilon_n}{p_n}j)-\sum_{j \in I'}\eta'_j \exp(\frac{\varepsilon_n}{p_n}j) =0
\]
But Lemma \ref{Hermite} tell us that $e^{\ds \varepsilon_n/p_n}$ is a transcendental number. This clearly 
forces $I=I'$ and 
the proof of the lemma is complete. 
\end{proof}

\begin{proof}[Proof of Theorem \ref{Gauss}] Let $A$ be a Borel set with
$\mu_s(A)>0$ and notice that for any positive integer $n$, we have
$$\bigintss_{\R}\Big|\frac{\sqrt{2}}{\sqrt{p_n}}
\sum_{j=0}^{p_n-1} \cos(\omega_n(j)t)\Big |^2 d\mu_s(t) \leq 1.$$
Therefore, applying the Helly theorem we may assume that the
sequence $$\Big(\frac{\sqrt{2}}{\sqrt{p_n}}
\sum_{j=0}^{p_n-1} \cos(\omega_n(j)t)\Big)_{n \geq 0}$$ converge in distribution. As is well-known, it is sufficient to show that for every real
number $x$,

\begin{eqnarray}{\label {eq:eq7}}
  \displaystyle \frac1{\mu_s(A)}\bigintss_A \exp\left \{-ix\frac{\sqrt{2}}{\sqrt{p_n}}
  \sum_{j=0}^{p_n-1} \cos(\omega_n(j)t) \right \} d\mu_s(t)
  \tend{n}{\infty} \exp(-\frac{x^2}2). \nonumber
\end{eqnarray}

\noindent{} To this end we apply theorem \ref{Mcleish} in the following
context. The measure space is the given Borel set $A$ of positive 
measure with respect to the probability measure $\mu_s$ on $\R$ equipped with the normalised measure 
$\ds \frac{\mu_s}{\mu_s(A)}$ and the random
variables are given by
$$
X_{nj}=\frac{\sqrt{2}}{\sqrt{p_n}} \cos(\omega_n(j)t),~~~~{\rm {where}}~~~0 \leq j \leq p_n-1,~
n \in \N.
$$

\noindent{}It is easy to check that the variables $\{X_{nj}\}$
satisfy condition (4). Further, condition (3) follows from the fact that
$$
\bigintss_{\R} \Big |\sum_{j=0}^{p_n-1} X_{nj}^2-1 \Big |^2 d\mu_s(t)
\tend{n}{\infty} 0.
$$
\noindent{} It remains to verify conditions (1) and (2) of Theorem \ref{Mcleish}. For this purpose, we set 
\begin{eqnarray*}
\Theta_n(x,t)&=&\Prod_{j=0}^{p_n-1}
\Big(1-ix\frac{\sqrt{2}}{\sqrt{p_n}}\cos(\omega_n(j)t\Big)\\
&=&1+\sum_{w \in W_n}{\rho_w}^{(n)}(x) \cos(wt), \\
\end{eqnarray*}
\noindent{} and
$$ \W_n= \bigcup_{r} \W_n^{(r)},$$
\noindent{}where $\W_n^{(r)}$ is the set of words of length $r$. Hence
\begin{eqnarray*}
\Big |\Theta_n(x,t)\Big|  \leq
\left \{\prod_{j=0}^{p_n-1}\Big ( 1+\frac{2x^2}{p_n}\Big) \right
\}^{\frac12}.
\end{eqnarray*}

\noindent{}But, since $1+u \leq e^u$, we get

\begin{eqnarray}
\Big|\Theta_n(x,t)\Big| \leq e^{x^2}.
\end{eqnarray}

\noindent{}This shows that the condition (1) is satisfied. It still remains to prove     
that the variables $\{X_{nj}\}$
satisfy condition (2). For that, it is sufficient to show that
\begin{eqnarray}\label{eq:eq11}
\bigintss_A \Prod_{j=0}^{p_n-1}\left (
1-ix\frac{\sqrt{2}}{\sqrt{p_n}}\cos(\omega_n(j)t)\right)d\mu_s(t)
\tend{n}{\infty}\mu_s(A).
\end{eqnarray}







\noindent{}Observe that

$$\bigintss_{A} \Theta_n(x,t) d\mu_s(t) =
  \mu_s(A)+\sum_{w \in \W_n}{\rho_w}^{(n)}(x) \bigintss_{A} 
  \cos(wt)d\mu_s(t)$$

\noindent{}and for
$w =\sum_{j=1}^{r}\omega_n(q_j)\in \W_n$, we have

\begin{eqnarray}\label{nicemajoration}
|{\rho_w}^{(n)}(x)| \leq \frac{2^{1-r}|x|^r}{p_n^{\frac{r}2}},
\end{eqnarray}

\noindent{}hence
\[
\max_{w \in \W_n}|{\rho_w}^{(n)}(x)| \leq
\frac{|x|}{p_n^{\frac{1}2}}\tend{n}{\infty}0.
\]
We claim that it is sufficient to prove the following

\begin{eqnarray}{\label{density}}
\bigintss_{\R} \phi \Prod_{j=0}^{p_n-1} \Big
(1-ix\frac{\sqrt{2}}{\sqrt{p_n}}\cos(\omega_n(j)t)\Big) dt
\tend{n}{\infty}\bigintss_{\R} \phi dt,
\end{eqnarray}

\noindent{}for any function $\phi$ with compactly supported
Fourier transforms. Indeed, assume
that (\ref{density}) holds and let $\epsilon>0$.
Then, by the density of the functions with compactly supported
Fourier transforms \cite[p.126]{Katznelson}, one can find a
function $\phi_{\epsilon}$ with compactly supported
Fourier transforms  such that
\[
\Big\|\chi_A.K_s-\phi_\epsilon\Big\|_{L^1(\R)} <\epsilon,
\] 
where
\noindent{}$\chi_A$  is indicator function of $A$.
\noindent{}Hence, according to (\ref{density}) combined with \eqref{eq:eq11}, for $n$
sufficiently large, we have

\begin{eqnarray}{\label{eq :eq12}}
&&\Big |\bigintss_A \Theta_n(x,t) d\mu_s(t) -\mu_s(A)\Big|
=\Big|\bigintss_A\Theta_n(x,t)d\mu_s(t) -\bigintss_{\R} \Theta_n(x,t) \phi_\epsilon(t) dt+
\\&&\bigintss_{\R} \Theta_n(x,t) \phi_\epsilon(t) dt-\bigintss_{\R}  \phi_\epsilon(t) dt + 
\bigintss_{\R}
\phi_\epsilon(t) dt-\mu_s(A)|< e^{x^2}\epsilon+2\epsilon.
\end{eqnarray}

\noindent{}The proof of the claim is complete. It still remains to prove
(\ref{density}). For that, let us compute the cardinality of words of length $r$ 
which can belong to the support of $\phi$.\\ 

\noindent{}By the well-known sampling theorem, we can assume that 
the support of $\phi$ is $[-\Omega,\Omega]$, $\Omega>0$. First, it is easy to check  that 
 for all odd $r$, $|w_n^{(r)}| \tendn +\infty$. It suffices to consider the words with even length. The case
$r=2$ is easy, since it is obvious to obtain the same conclusion. Moreover, as we will see later, 
it is sufficient to consider the case $r=2^k$, $k \geq 2$. We argue that the cardinality of words of length $2^k$, 
$k \geq 2$ which can belong to $[-\Omega,\Omega]$ is less than $\displaystyle \Omega.\frac{p_n^k~{(\log(p_n))}^{k-1}}{m_n~ \varepsilon_n^{k-2}}$. Indeed,
for $k=2$. Write 
$$w_n^{(4)}=\eta_1\omega_n(k_1)+\eta_2\omega_n(k_2)+\eta_3\omega_n(k_3)+\eta_4.\omega_n(k_4),
$$
with $\eta_i \in \{-1,1\}$ and $k_i \in \{0,\cdots,p_n-1\},~~i=1,\cdots,4,$ 
and put 
$$e_n(p)=\exp\big(\frac{\varepsilon_n}{p_n}.p\big),~~p\in \{0,\cdots,p_n-1\}.$$
If $\sum_{i=1}^{4}\eta_i \neq 0$ then there is nothing to prove since $\big|w_n^{(4)}\big| \tendn +\infty$. Therefore, 
let us assume that  $\sum_{i=1}^{4}\eta_i = 0.$  In this case, without loss of generality (WLOG), we will assume that 
$k_1 <k_2<k_3<k_4$. Hence 
\begin{eqnarray*}
w_n^{(4)}&=&\eta_1\omega_n(k_1)+\eta_2\omega_n(k_2)+\eta_3\omega_n(k_3)+\eta_4\omega_n(k_4)\\
&=& \displaystyle \frac{m_n}{\epsilon_n^2}p_n e_n(k_1)\Big(\eta_1+\eta_2e_n(\alpha_1)+
\eta_3e_n(\alpha_2)+\eta_4e_n(\alpha_3)\Big)
\end{eqnarray*}
where, $\alpha_i=k_{i+1}-k_1$, $i=1,\cdots,3.$ At this stage, we may assume again WLOG that 
$\eta_1+\eta_2=0$ and $\eta_3+\eta_4=0$. It follows that
\begin{eqnarray*}
w_n^{(4)}= \displaystyle \frac{m_n}{\epsilon_n^2}p_n e_n(k_1)\Big(\eta_2\big(e_n(\alpha_1)-1\big)+
\eta_4 e_n(\alpha_2)\big(e_n(\alpha'_1)-1\big)\Big) {\rm {~~with~~}} \alpha'_1=\alpha_3-\alpha_2.
\end{eqnarray*}
Consequently, we have two cases to deal with. 
\begin{itemize}
 \item {\bf Case 1:} $\eta_2=\eta_4$. Then 
\begin{eqnarray*}
 \Big| w_n^{(4)} \Big| &\geq& \displaystyle \frac{m_n}{\epsilon_n^2}p_n e_n(k_1) \Big( e_n(\alpha_1)-1 \Big) \\
&\geq& \displaystyle \frac{m_n}{\epsilon_n^2}p_n \frac{\varepsilon_n}{p_n} \tendn +\infty.
\end{eqnarray*}
\item  {\bf Case 2:} $\eta_2=-\eta_4.$ We thus get 
\begin{eqnarray*}
 w_n^{(4)}&=& \displaystyle \frac{m_n}{\epsilon_n^2}p_n.e_n(k_1) \eta_4 \Big(
e_n(\alpha_2)\big(e_n(\alpha'_1)-1\big)-\big(e_n(\alpha_1)-1\big)\Big)
\end{eqnarray*}
Hence 
\begin{eqnarray*}
 \big|w_n^{(4)}\big|= \displaystyle \frac{m_n}{\epsilon_n^2}p_n e_n(k_1)\Big| 
e_n(\alpha_2)\big(e_n(\alpha'_1)-1\big)-\big(e_n(\alpha_1)-1\big)\Big|
\end{eqnarray*}   
Therefore, we have three cases to deal with.
\begin{itemize}
 \item {\bf Case 1:} $\alpha'_1>\alpha_1$. In this case, 
\begin{eqnarray*}
\big|w_n^{(4)}\big| &=& \displaystyle \frac{m_n}{\epsilon_n^2}p_n e_n(k_1)\Big( 
e_n(\alpha_2)\big(e_n(\alpha'_1)-1\big)-\big(e_n(\alpha_1)-1\big)\Big)\\
&\geq& \displaystyle 
\frac{m_n}{\epsilon_n^2}p_n \big(e_n(\alpha'_1)-e_n(\alpha_1)\big)\\
&\geq& \displaystyle 
\frac{m_n}{\epsilon_n^2}p_n \big(e_n(\alpha'_1-\alpha_1)-1\big) \tendn +\infty.
\end{eqnarray*} 
\item {\bf Case 2:} $\alpha'_1<\alpha_1$. Write $\alpha_1=\alpha'_1+\beta$. Thus
\begin{eqnarray*}
\big|w_n^{(4)}\big| &=& \displaystyle \frac{m_n}{\varepsilon_n^2}p_n e_n(k_1) \Big| 
e_n(\alpha_2)\big(e_n(\alpha'_1)-1\big)-\big(e_n(\alpha'_1+\beta)-1\big)\Big|\\
&\geq& \displaystyle 
\frac{m_n}{\varepsilon_n^2}p_n \Big|e_n(\alpha'_1+\alpha_2)-e_n(\alpha_2)-e_n(\alpha'_1+\beta)+1\Big|   \\
&\geq& \frac{m_n}{\varepsilon^2_n}p_n \Big| \big(e_n(\alpha_2)-1\big)\big(e_n(\alpha'_1)-1\big)-
e_n(\alpha'_1)\big(e_n(\beta)-1\big)
\Big|\\
&\geq& \frac{m_n}{\varepsilon^2_n}p_n {\frac{\varepsilon_n\beta \alpha'_1}{p_n}} \Big( 
\frac{e_n(\alpha'_1)\big(e_n(\beta)-1\big)}{\frac{\varepsilon_n\beta \alpha'_1}{p_n}}-
{\frac{\big(e_n(\alpha_2)-1\big)\big(e_n(\alpha'_1)-1\big)}{\frac{\varepsilon_n\beta \alpha'_1}{p_n}}}
\Big) \\
&\geq& \frac{m_n}{\varepsilon_n}\beta \alpha'_1 \Big( 
\frac{e_n(\alpha'_1)\big(e_n(\beta)-1\big)}{\frac{\varepsilon_n\beta \alpha'_1}{p_n}}-
{\frac{\big(e_n(\alpha_2)-1\big)\big(e_n(\alpha'_1)-1\big)}{\frac{\varepsilon_n\beta \alpha'_1}{p_n}}}
\Big) 
\end{eqnarray*}  
\noindent{}But for any $x \in [0,\log(2)[$, we have $ x \leq e^x-1 \leq 2x$. Therefore
$${\frac{\big(e_n(\alpha_2)-1\big)\big(e_n(\alpha'_1)-1\big)}{\frac{\varepsilon_n\beta \alpha'_1}{p_n}}} 
\leq 4 \varepsilon_n.
$$
Hence
$$
\big|w_n^{(4)}\big| \tendn +\infty
$$
\item {\bf Case 3:} $\alpha'_1=\alpha_1$. In this case, we get
$$
\big|w_n^{(4)}\big| \geq \displaystyle \frac{m_n}{\epsilon_n^2}p_n \Big( 
e_n(\alpha_2)-1\Big)\Big(e_n(\alpha_1)-1\Big).
$$   
From this, we have
\begin{eqnarray*} 
\big|w_n^{(4)}\big| &\geq& \displaystyle \frac{m_n}{\epsilon_n^2}p_n.\Big(\frac{\varepsilon_n}{p_n}\Big)^2.
\alpha_2.\alpha_1.\\
&\geq& \displaystyle \frac{m_n}{p_n} \alpha_2.\alpha_1
\end{eqnarray*} 
\end{itemize}
It follows that 
$$\big|w_n^{(4)}\big| \leq \Omega \Longrightarrow \alpha_2.\alpha_1 \leq \frac{p_n}{m_n}.\Omega.$$ 
But the cardinality of $(\alpha_1,\alpha_2)$ such that 
$\displaystyle \alpha_2.\alpha_1 \leq \Omega \frac{p_n}{m_n}$ is less than $\displaystyle \Omega \frac{p_n}{m_n} \log(p_n).$ This gives that the 
cardinality of words of length $4$ which can belong to $[-\Omega,\Omega]$ is less than  
$\displaystyle \Omega \frac{p_n^2}{m_n} \log(p_n).$      
\end{itemize}
Repeated the same argument as above we deduce that the only words to take into account at the stage $k=3$ are the form
$$w_n^{(8)}=\frac{m_n}{\varepsilon^2_n}p_ne_n(\alpha)\Big((e_n(\alpha_1)-1)(e_n(\alpha_2)-1)+\eta.
e_n(\alpha_3)(e_n(\alpha_4)-1)(e_n(\alpha_5)-1)\Big),$$
where $\eta=\pm 1$. In the case $\eta=1$, it is easy to see that 
\begin{eqnarray*}
 \big|w_n^{(8)} \big|\geq \frac{m_n \varepsilon_n}{p^2_n}\alpha_3  \widetilde{\alpha_1} \widetilde{\alpha_2}, 
{\textrm {~~where~~}}  \widetilde{\alpha_i}=\inf(\alpha_i,\alpha_{3+i}).
\end{eqnarray*}
For $\eta=-1$, write
\begin{eqnarray*}
\big|w_n^{(8)}\big|&=&\frac{m_n}{\varepsilon^2_n}p_ne_n(\alpha)\Big|(e_n(\alpha_1)-1)(e_n(\alpha_2)-1)-\\
&&(e_n(\alpha_3)-1)(e_n(\alpha_4)-1)(e_n(\alpha_5)-1)-
(e_n(\alpha_4)-1)(e_n(\alpha_5)-1)\Big|.
\end{eqnarray*}
Using the following expansion 
\begin{eqnarray}\label{eq:eq13}
e_n(x)=1+\frac{\varepsilon_n .x}{p_n}+o(1),
\end{eqnarray}
we obtain, for a large $n$, 
\begin{eqnarray*}
 \big|w_n^{(8)} \big|\geq \frac{m_n \varepsilon_n}{p^2_n}\alpha_3  \widetilde{\alpha_1} \widetilde{\alpha_2}.
\end{eqnarray*}
We deduce that the cardinality of words of length $8$ 
which can belong to $[-\Omega,\Omega]$ is less than   
$$\Omega .p_n.\frac{p^2_n}{m_n \varepsilon_n}\sum_{\widetilde{\alpha_1}, \widetilde{\alpha_2}}
\frac1{\widetilde{\alpha_1} \widetilde{\alpha_2}}
\leq \Omega. \frac{p_n^{3}}{m_n \varepsilon_n} (\log(p_n))^{2} .$$
In the same manner as before consider the words of length $k$ in the following form 
$$w_n^{(2^k)}= \displaystyle \frac{m_n}{\epsilon_n^2}p_ne_n(\alpha)\Big((e_n(\alpha_1)-1) \cdots (e_n(\alpha_k)-1)\Big).$$
Therefore
\begin{eqnarray*} 
\big|w_n^{(2^k)}\big| &\geq& \displaystyle \frac{m_n}{\epsilon_n^2}p_n.\Big(\frac{\varepsilon_n}{p_n}\Big)^k.
\alpha_1.\alpha_2\cdots \alpha_k.\\
&\geq& \displaystyle \frac{m_n}{p_n^{k-1}} \varepsilon^{k-2}_n  \alpha_1.\alpha_2\cdots \alpha_k
\end{eqnarray*}
which yields as above that the cardinality of words of length $2^k$ 
which can belong to $[-\Omega,\Omega]$ is less than   
$$\Omega .p_n.\frac{p_n^{k-1}}{m_n \varepsilon^{k-2}_n}\sum_{\alpha_2,\cdots,\alpha_k}\frac1{\alpha_2\cdots\alpha_k}
\leq \Omega. p_n.\frac{p_n^{k-1}}{m_n \varepsilon^{k-2}_n} (\log(p_n))^{k-1} .$$
Now, if $r$ is any arbitrary even number. Write $r$ in base 2 as 
$$r=2^{l_s}+\cdots+2^{l_1} {\textrm {~~with~~}} l_s > \cdots > l_1 \geq 1.$$ 
and write
$$w_n^{(r)}=w_n^{(2^{l_s})}+\cdots+w_n^{(2^{l_1})},$$
with
$$ 
w_n^{(2^{l_j})}=\eta^{(j)}_1\omega_n(k_1^{(j)})+\cdots+\eta^{(j)}_{2^{l_j}}\omega_n(k_{l_j}^{(j)}),~~j=1,\cdots,s,$$
and 
$$ \sum_{i=1}^{2^{l_j}}\eta^{(j)}_i=0,~~j=1,\cdots,s.$$
Observe that the important case to consider is the case
$$w_n^{(2^{l_j})}=\pm\frac{m_n}{\varepsilon^2_n}p_n\prod_{i=1}^{l_j}\big(e(\alpha_i^{(j)})-1\big), ~j=1,\cdots,s.$$
Using again \eqref{eq:eq13} we obtain for a large $n$ that  
$$\big|w_n^{(r)}\big| \geq \displaystyle \frac{m_n}{p_n^{l_s-1}} \varepsilon^{l_s-2}_n 
\prod_{i=1}^{l_s}\alpha^{(s)}_i. $$
Hence, the cardinality of 
words of length $r$ which can belong to $[-\Omega,\Omega]$ is less than 
$$\Omega.\frac1{m_n}.\frac{p_n^{\lfloor \log_2(r)\rfloor}}{\varepsilon^{\lfloor \log_2(r)\rfloor-2}_n} 
(\log(p_n))^{\lfloor \log_2(r)\rfloor-1},$$
and, in consequence, 
by \eqref{nicemajoration}, we deduce that 
\begin{eqnarray*}
\displaystyle \sum_{w \in\W_n} \Big| {\rho_w}^{(n)}(x)  \int_{\R} e^{-iwt}
\phi(t)dt\Big| \leq \displaystyle \sum_{\overset{w \in\W_n^{(r)}, r {\textrm{~~even~~}}}
{4 \leq r \leq p_n}}\Big| {\rho_w}^{(n)}(x)\int_{\R} \phi(t) e^{-iwt} dt\Big|.
\end{eqnarray*}
Therefore, under the assumption (2), we get 
\begin{eqnarray*}
\displaystyle \sum_{w \in\W_n} \Big| {\rho_w}^{(n)}(x)  \int_{\R} e^{-iwt}
\phi(t)dt\Big| &\leq& \displaystyle \Omega \frac{|x|}{m_n} 
\sum_{\overset{r {\textrm{~~even~~}}}
{4 \leq r \leq p_n}} 
\frac{(\log(p_n))^{\lfloor \log_2(r)\rfloor -1}}{p_n^{\big(\frac{r}{2}-2\lfloor\log_2(r)\rfloor+2\big)}}  
.\frac1{{\big(p_n.\varepsilon_n\big)}^{\lfloor \log_2(r)\rfloor -2}}\\
&\leq& \displaystyle \Omega .|x|.\frac{\log(p_n)}{m_n} 
\sum_{\overset{r {\textrm{~~even~~}}}
{4 \leq r \leq p_n}}
\frac1{p_n^{\big(\frac{r}{2}-2\lfloor\log_2(r)\rfloor+2\big)}}.  
\end{eqnarray*}
The last inequality is due to the fact that for a large $n$ we may assume that 
$\displaystyle\frac{\log(p_n)}{m_n}$ is strictly less than 1. In addition, since $p_n \geq 2$,
$\displaystyle \sum_{\overset{r {\textrm{~~even~~}}}{4 \leq r \leq p_n}} 
\frac1{p_n^{\big(\frac{r}{2}-2\lfloor\log_2(r)\rfloor+2\big)}}$ is convergent. We conclude that 
\[
\Big|\sum_{w \in W_n}{\rho_w}^{(n)}(x) \bigintss_{\R} \phi(t)
  \cos(wt)dt|\leq \Omega .|x|.\frac{\log(p_n)}{m_n} 
\sum_{\overset{r {\textrm{~~even~~}}}
{4 \leq r \leq p_n}} 
\frac1{p_n^{\big(\frac{r}{2}-2\lfloor\log_2(r)\rfloor+2\big)}} \tend{n}{\infty}0.
\]

This finishes the proof, the other case is left to the reader.
\end{proof}
\begin{xrem}Let us point out that if
\begin{eqnarray}\label{factoriel} 
\frac{h_n.\varepsilon_n^k}{p_n^k} \tendn +\infty, {\textrm {~~for~any~}} k \geq 1.
\end{eqnarray}
Then the spectrum of the associated exponential staircase flow is singular. Indeed, Let 
$W_n^{(r)}=\displaystyle\frac{m_n.p_n}{\varepsilon^2_n}\sum_{j=1}^{r}\eta_j e_n(k_j)$ a word of length $r$. 
If $W_n^{(r)} \neq 0$. Then, by Lemma \ref{Hermite}, there exists $\alpha \geq 1$ such that 
$$\sum_{j=1}^{r}\frac{\eta_j k_j^{\alpha}}{\alpha!} \neq 0.$$ 
Therefore, using a Taylor expansion of $e_n$ and assuming that all the terms of degree less than $\alpha$ are 0, we have
\begin{eqnarray*}
|W_n^{(r)}| &=&\frac{m_n.p_n}{\varepsilon_n}\Big|
{\Big(\frac{\varepsilon_n}{p_n}\Big)}^{\alpha}\sum_{j=1}^{r}\frac{\eta_j k_j^{\alpha}}{\alpha!}
+o\Big({\Big(\frac{\varepsilon_n}{p_n}\Big)}^{\alpha}\sum_{j=1}^{r}\frac{\eta_j k_j^{\alpha}}{\alpha!}\Big)\Big|\\
 &\geq& \Big|\frac{h_n.\varepsilon^{\alpha-1}_n}{p_n^{\alpha-1}}\sum_{j=1}^{r}\frac{\eta_j k_j^{\alpha}}{\alpha!}
+o\Big(\frac{h_n.\varepsilon^{\alpha-1}_n}{p_n^{\alpha-1}}\sum_{j=1}^{r}\frac{\eta_j k_j^{\alpha}}{\alpha!}\Big)\Big|
\tendn +\infty.
\end{eqnarray*}
One can apply Lagrangian method to show that $\alpha$ is strictly less than two. Furthermore,
to ensure that the assumption \eqref{factoriel} holds, it suffices to take $\displaystyle p_n=n, \varepsilon_n =\frac1{n^2}$ and
$\displaystyle m_n=\frac{h_n}{n^2}$.  
\end{xrem}
\begin{proof}[Proof of Proposition \ref{Gauss}]
Let $A$ be a Borel subset of $\R$, and $x \in ]1,+\infty[$, then, for any
positive integer $m$, we have

\begin{eqnarray*}{\label{eq :eqfin}}
\bigintss_A \Big|\big|P_m(\theta)\big|^2-1\Big| d\mu_s(\theta) &\geq& \bigintss_{\{\theta
\in A
~~~:~~~~|P_m(\theta)|>x\}}\Big|\big|P_m(\theta)\big|^2-1\Big| d\mu_s(\theta)\\
&\geq&(x^2-1)\mu_s\Big\{\theta \in A~:~|P_m(\theta)|>x \Big\}\\
&\geq& (x^2-1) \mu_s\Big\{\theta \in A ~~:~~|\Re({P_m(\theta)})|>x\Big\}
\end{eqnarray*}
\noindent{} Let $m$ goes to infinity and use Theorem
\ref{Gauss} and Proposition \ref{w-lim} to get

\[
\bigintss_A \phi\; d\mu_s \geq (x-1)\{1-{\mathcal
{N}}([-\sqrt{2}x,\sqrt{2}x])\}\mu_s(A).
\]

\noindent{}Put $K=(x-1)\{1-{\mathcal {N}}([-\sqrt{2}x,\sqrt{2}x])\}$. Hence 
$$
\bigintss_A \phi\; d\mu_s \geq K \mu_s(A),$$ 
\noindent{}for any Borel subset $A$ of $\R$. This end the proof of the proposition.
\end{proof}

\noindent{}Now, we give the proof of our main result.\\

\begin{proof}[Proof of Theorem \ref{main}] Follows easily from the proposition \ref{key-sasha} combined with
proposition \ref{CDsuffit}.
\end{proof}
\begin{xrem}It is shown in \cite{elabdal-archiv} that the spectrum of the Ornstein rank one flows is singular and in 
the forthcoming papers we shall extended the classical results of Klemes \cite{Klemes1} and 
Klemes-Reinhold \cite{Klemes2}. This allow us to ask the following question 

\begin{que}Does any rank one flow have singular spectrum? 
\end{que} 
\end{xrem}

\begin{thank}\em
The author wishes to express his thanks to J-P. Thouvenot, B. Host, Mariusz Lemanczyk, Emmanuel Lesigne, 
T. de la Rue, A. A. Prikhod'ko and Margherita Disertori for fruitful discussions on the subject. He would like 
to express his special thanks to Bassam Fayad for pointing out a gap in the previous version of this paper.
\end{thank}

\end{document}